\documentclass[11pt,a4paper]{article}
  \def\SvZfontcode{8} 
  \def\SvZslantedGreekCapitals{1}
  \usepackage{amsmath,amsthm,amscd}

\usepackage[T1]{fontenc}

\def\SvZrequireslantedRedef{0}
\ifcase\SvZfontcode
\usepackage{bm}
\usepackage{amssymb}
\def\SvZrequireslantedRedef{1}
\or
\usepackage{fouriernc}
\usepackage{bm}
\def\SvZrequireslantedRedef{1}
\or
\usepackage{fourier}
\usepackage{bm}
\def\SvZrequireslantedRedef{1}
\or
\usepackage{amssymb}
\if\SvZslantedGreekCapitals 1
\usepackage[slantedGreek]{mathptmx}
\else
\usepackage{mathptmx}
\fi
\DeclareMathAlphabet{\bm}{OT1}{ptm}{b}{it} 
\or
\usepackage{kmath,kerkis}
\usepackage{bm}
\def\SvZrequireslantedRedef{1}
\or
\if\SvZslantedGreekCapitals 1
\usepackage[sc,slantedGreek]{mathpazo}
\else
\usepackage[sc]{mathpazo}
\fi
\usepackage{bm}
\or
\usepackage[adobe-utopia]{mathdesign}
\usepackage{bm}
\or
\usepackage[urw-garamond]{mathdesign}
\usepackage{bm}
\or
\usepackage[bitstream-charter]{mathdesign}
\usepackage{bm}
\or
\usepackage{lmodern}
\usepackage{bm}
\usepackage{amssymb}
\def\SvZrequireslantedRedef{1}
\or
\usepackage{arev}
\fi

\if\SvZrequireslantedRedef 1
\if\SvZslantedGreekCapitals 1
\renewcommand{\Gamma}{\varGamma}
\renewcommand{\Delta}{\varDelta}
\renewcommand{\Theta}{\varTheta}
\renewcommand{\Lambda}{\varLambda}
\renewcommand{\Xi}{\varXi}
\renewcommand{\Pi}{\varPi}
\renewcommand{\Sigma}{\varSigma}
\renewcommand{\Upsilon}{\varUpsilon}
\renewcommand{\Phi}{\varPhi}
\renewcommand{\Psi}{\varPsi}
\renewcommand{\Omega}{\varOmega}
\fi
\fi

\renewcommand{\phi}{\varphi}

\reversemarginpar

\usepackage{url}
\usepackage{graphicx}
\usepackage{color}

\ifx\SvZinPresentation\undefined
\usepackage{float}
\usepackage{longtable}
\usepackage[margin=10pt,labelfont=bf,textfont=it]{caption}

\usepackage[colorlinks]{hyperref}

\fi

\newcommand{\ignore}[1]{}

\let\Oldsetminus\setminus
\renewcommand{\setminus}{\ensuremath{-}}

\newcommand{\delete}{\ensuremath{\!\Oldsetminus\!}}
\newcommand{\contract}{\ensuremath{\!/}}

\ifx\SvZinPresentation\undefined
\newtheorem{theorem}{Theorem}[section]
\else
\newtheorem{theorem}{Theorem}
\fi

\newtheorem{lemma}[theorem]{Lemma}

\newtheorem{corollary}[theorem]{Corollary}

\newenvironment{claimenv}{\list{}{\rightmargin0pt\leftmargin10pt\topsep0pt}\item[]}{\endlist}

  \newcommand{\XX}{\ensuremath{\mathit{X\!X}}}
  \newcommand{\trugr}{\ensuremath{\ddot{U}}}

	\usepackage[numbers]{natbib}
	\newcommand{\Dutchvon}[2]{#2}

\begin{document}
\title{The structure of graphs with a vital linkage of order 2\thanks{The research of all authors was partially supported by a grant from the Marsden Fund of New Zealand. The first author was also supported by a FRST Science \& Technology post-doctoral fellowship. The third author was also supported by the Netherlands Organization for Scientific Research (NWO).}}
\author{Dillon Mayhew\thanks{School of Mathematics, Statistics and Operations Research, Victoria University of Wellington, New Zealand. E-mail: \url{Dillon.Mayhew@msor.vuw.ac.nz}, \url{Geoff.Whittle@msor.vuw.ac.nz}} \and Geoff Whittle\footnotemark[2]  \and Stefan H. M. van Zwam \thanks{Centrum Wiskunde en Informatica, Postbus 94079, 1090 GB Amsterdam, The Netherlands. E-mail: \url{Stefan.van.Zwam@cwi.nl}}}

\maketitle

\abstract{A \emph{linkage of order $k$} of a graph $G$ is a subgraph with $k$ components, each of which is a path. A linkage is \emph{vital} if it spans all vertices, and no other linkage connects the same pairs of end vertices. We give a characterization of the graphs with a vital linkage of order 2: they are certain minors of a family of highly structured graphs.}

\section{Introduction}

\citet{RSXXI} defined a \emph{linkage} in a graph $G$ as a subgraph in which each component is a path. The \emph{order} of a linkage is the number of components. A linkage $L$ of order $k$ is \emph{unique} if no other collection of paths connects the same pairs of vertices, it is \emph{spanning} if $V(L) = V(G)$, and it is \emph{vital} if it is both unique and spanning. Graphs with a vital linkage are well-behaved. For instance, Robertson and Seymour proved the following:

\begin{theorem}[{\citet[Theorem 1.1]{RSXXI}}]
  There exists an integer $w$, depending only on $k$, such that every graph with a vital linkage of order $k$ has tree width at most $w$.
\end{theorem}

Note that Robertson and Seymour use the term $p$-linkage to denote a linkage with $p$ terminals. Robertson and Seymour's proof of this theorem is surprisingly elaborate, and uses their structural description of graphs with no large clique-minor. Recently \citet{KW10} proved a strengthening of this result. Their shorter proof avoids using the structure theorem.

Our interest in linkages, in particular those of order 2, stems from quite a different area of research: matroid theory. \citet{TruVI} studied a class of binary matroids that he calls \emph{almost regular}. His proofs lean heavily on a class of matroids that are single-element extensions of the cycle matroids of graphs with a vital linkage of order 2. These matroids turned up again in the excluded-minor characterization of matroids that are either binary or ternary, by \citet{MOOW11}. 

Truemper proves that an almost regular matroid can be built from one of two specific matroids by certain $\Delta-Y$ operations. 
This is a deep result, but it does not yield bounds on the branch width of these matroids. In a forthcoming paper the authors of this paper, together with Chun, will give an explicit structural description of the class of almost regular matroids \cite{CMWZ10}. The main result of this paper will be of use in that project.


To state our main result we need a few more definitions. Fix a graph $G$ and a spanning linkage $L$ of order $k$. A \emph{path edge} is a member of $E(L)$; an edge in $E(G)\setminus E(L)$ is called a \emph{chord} if its endpoints lie in a single path, and a \emph{rung edge} otherwise. If $L$ is vital, then $G$ cannot have any chords.

A \emph{linkage minor} of $G$ with respect to a (chordless) linkage $L$ is a minor $H$ of $G$ such that all path edges in $E(G)\setminus E(H)$ have been contracted, and all rung edges in $E(G)\setminus E(H)$ have been deleted. If the linkage $L$ is clear from the context we simply say that $H$ is a linkage minor of $G$. 
Moreover, let $G$ be a graph with a chordless 2-linkage $L$. If $G$ has a linkage minor isomorphic to $K_{2,4}$, such that the terminals of $L$ are mapped to the degree-2 vertices of $K_{2,4}$, we say that $G$ has an $\XX$ linkage minor (cf. Figure \ref{fig:XX}).

\begin{figure}[tbp]
  \centering
  \includegraphics{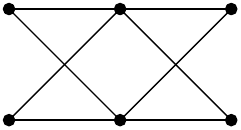}
  \caption{The graph $K_{2,4}$.}\label{fig:XX}
\end{figure}

\begin{figure}[tbp]
  \centering
  \includegraphics{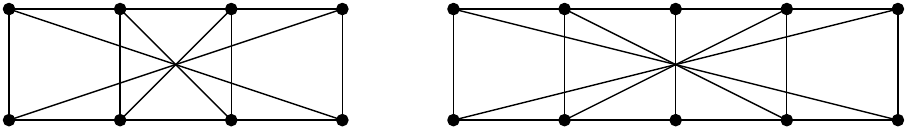}
  \caption{The graphs $\trugr_4$ and $\trugr_5$.}\label{fig:trugr}
\end{figure}

For each integer $n$, the graph $\trugr_n$ is the graph with $V(\trugr_n) = \{v_1, \ldots, v_{n}\} \cup \{u_1, \ldots, u_{n}\}$, and
\begin{align}
  E(\trugr_n) = & \{ v_i v_{i+1} \mid i = 1, \ldots, n-1\} \cup \{u_i u_{i+1} \mid i = 1, \ldots, n-1\} \cup\notag\\
  & \{ u_i v_i \mid i = 1, \ldots, n\} \cup \{u_i v_{n+1-i} \mid i = 1, \ldots, n\}.
\end{align}
We denote by $L_n$ the linkage of $\trugr_n$ consisting of all edges $v_iv_{i+1}$ and $u_iu_{i+1}$ for $i = 1, \ldots, n-1$. In Figure \ref{fig:trugr} the graphs $\trugr_4$ and $\trugr_5$ are depicted.

Finally, we say that $G$ is a \emph{Truemper graph} if $G$ is a linkage minor of $\trugr_n$ for some $n$. The main result of this paper is the following:

\begin{theorem}\label{thm:mainresult}
  Let $G$ be a graph. The following statements are equivalent:
  \begin{enumerate}
    \item\label{main:i}   $G$ has a vital linkage of order 2;
    \item\label{main:ii}  $G$ has a chordless spanning linkage of order 2 with no $\XX$ linkage minor;
    \item\label{main:iii} $G$ is a Truemper graph.
  \end{enumerate}
\end{theorem}

\citet{RSXXI} commented, without proof, that graphs with a vital linkage with $k\leq 5$ terminal vertices have path width at most $k$. A weaker claim is the following:

\begin{corollary}
	Let $G$ be a graph with a vital linkage of order $2$. Then $G$ has path width at most 4.
\end{corollary}

Another consequence of our result is that graphs with a vital linkage of order 2 embed in the projective plane:

\begin{corollary}
  Let $G$ be a graph with a vital linkage of order 2. Then $G$ can be embedded on a M\"obius strip.
\end{corollary}

Both corollaries can be seen to be true by considering an alternative depiction of $\trugr_{2n}$, analogous to Figure \ref{fig:spiderweb}. 

\begin{figure}[tbp]
  \centering
  \includegraphics[scale=.9]{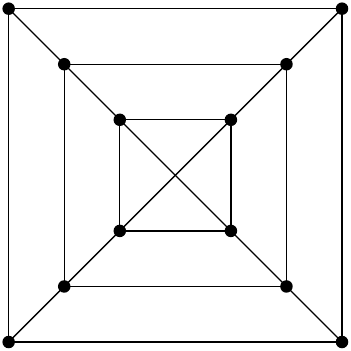}
  \caption{The graph $\trugr_6$. The linkage is formed by the two diagonally drawn paths.}\label{fig:spiderweb}
\end{figure}

\section{Proof of Theorem \ref{thm:mainresult}}
We start with a few more definitions. Suppose $L$ is a linkage of order 2 with components $P_1$ and $P_2$, such that the terminal vertices of $P_1$ are $s_1$ and $t_1$, and those of $P_2$ are $s_2$ and $t_2$. We order the vertices on the paths in a natural way, as follows. If $v$ and $w$ are vertices of $P_i$, then we say that $v$ is \emph{(strictly) to the left} of $w$ if the graph distance from $s_i$ to $v$ in the subgraph $P_i$ is (strictly) smaller than the graph distance from $s_i$ to $w$. The notion \emph{to the right} is defined analogously.

We will frequently use the following elementary observation, whose proof we omit.

\begin{lemma}\label{lem:order}
  Let $G$ be a graph with a chordless spanning linkage $L$ of order 2. Let $P_1$ and $P_2$ be the components of $L$, with terminal vertices respectively $s_1,t_1$ and $s_2,t_2$. Let $H$ be a linkage minor of $G$. If $v$ and $w$ are on $P_i$, and $v$ is to the left of $w$, then the vertex corresponding to $v$ in $H$ is to the left of the vertex corresponding to $w$ in $H$.
\end{lemma}

Without further ado we dive into the proof, which will consist of a sequence of lemmas. The first deals with the equivalence of the first two statements in the theorem.

\begin{lemma}\label{lem:i-ii}
  Let $G$ be a graph with a chordless spanning linkage $L$ of order 2. Then $L$ is vital if and only if $G$ has no $\XX$ linkage minor with respect to $L$.
\end{lemma}

\begin{figure}[tbp]
  \centering
  \includegraphics{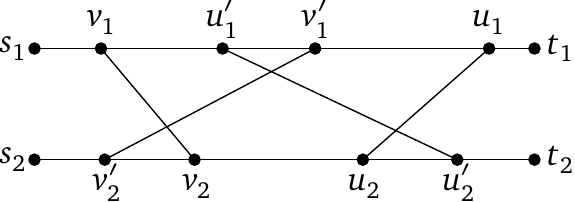}
  \caption{Detail of the proof of Lemma \ref{lem:i-ii}.}\label{fig:subpaths}
\end{figure}


\begin{proof}
  First we suppose that there exists a graph $G$ with a non-vital chordless spanning linkage $L$ of order 2 such that $G$ has no $\XX$ linkage minor. Let $P_1$, $P_2$ be the paths of $L$, where $P_1$ runs from $s_1$ to $t_1$, and $P_2$ runs from $s_2$ to $t_2$. Let $P_1'$, $P_2'$ be different paths connecting the same pairs of vertices. Without loss of generality, $P_1' \neq P_1$. But then $P_1'$ must meet $P_2$, so $P_2' \neq P_2$. Let $e = v_1v_2$ be an edge of $P_1'$ such that the subpath $s_1-v_1$ of $P_1'$ is also a subpath of $P_1$, but $e$ is not an edge of $P_1$. Let $f = u_2u_1$ be an edge of $P_1'$ such that the subpath $u_1-t_1$ of $P_1'$ is also a subpath of $P_2$, but $f$ is not an edge of $P_2$. Similarly, let $e' = v_2'v_1'$ be an edge of $P_2'$ such that the subpath $s_2-v_2'$ of $P_2'$ is also a subpath of $P_2$, but $e'$ is not an edge of $P_2$. Let $f' = u_1'u_2'$ be an edge of $P_2'$ such that the subpath $u_2'-t_2$ of $P_2'$ is also a subpath of $P_2$, but $f'$ is not on $P_2$. See Figure \ref{fig:subpaths}.

  Since $P_1'$ and $P_2'$ are vertex-disjoint, $v_2'$ must be strictly to the left of $v_2$ and $u_2$. For the same reason, $v_1'$ must be strictly between $v_1$ and $u_1$. Likewise, $u_2'$ must be strictly to the right of $v_2$ and $u_2$, and $u_1'$ must be strictly between $v_1$ and $u_1$. Now construct a linkage minor $H$ of $G$, as follows.  Contract all edges on the subpaths $s_1-v_1$, $v_1'-u_1'$, and $u_1-t_1$ of $P_1$, contract all edges on the subpaths $s_2-v_2'$, $v_2-u_2$, and $u_2'-t_2$ of $P_2$, delete all rung edges but $\{e,f,e',f'\}$, and contract all but one of the edges of each series class in the resulting graph. Clearly $H$ is isomorphic to $\XX$, a contradiction.

Conversely, suppose that $G$ has an $\XX$ linkage minor, but that $L$ is unique. Clearly having a vital linkage is preserved under taking linkage minors. But $\XX$ has two linkages, a contradiction.
\end{proof}

Next we show that the third statement of Theorem \ref{thm:mainresult} implies the second.

\begin{lemma}\label{lem:iii-ii}
  For all $n$, $\trugr_n$ has no $\XX$ linkage minor with respect to $L_n$.
\end{lemma}

\begin{proof}
  The result holds for $n \leq 2$, because then $|V(\trugr_n)| < |V(\XX)|$. Suppose the lemma fails for some $n \geq 3$, but is valid for all smaller $n$. Every edge of $\XX$ is incident with exactly one of the four end vertices of the paths. Hence all rung edges incident with at least two of the four end vertices are not in any $\XX$ linkage minor. But after deleting those edges from $\trugr_n$ the end vertices have degree one, and hence the edges incident with them will not be in any $\XX$ linkage minor. Contracting these four edges produces $\trugr_{n-2}$, a contradiction.
\end{proof}

%
\emph{Reversing a path} $P_i$ means exchanging the labels of vertices $s_i$ and $t_i$, thereby reversing the order on the vertices of the path. 

\begin{lemma}\label{lem:signtoTru}
  Let $G$ be a graph, and $L$ a chordless spanning linkage of order 2 of $G$ consisting of paths $P_1$, running from $s_1$ to $t_1$, and $P_2$, running from $s_2$ to $t_2$. 
If $G$ has no $\XX$ linkage minor, then $G$ is a linkage minor of $\trugr_n$ with respect to $L_n$ for some integer $n$, such that $L$ is a contraction of $L_n$. 
\end{lemma}


\begin{proof}
  Suppose the statement is false. Let $G$ be a counterexample with as few edges as possible. If some end vertex of a path, say $s_1$, has degree one (with $e = s_1v$ the only edge), then we can embed $G\contract e$ in $\trugr_n$ for some $n$. Let $G'$ be obtained from $\trugr_n$ by adding four vertices $s_1', t_1', s_2', t_2'$, and edges $s_1'v_1, s_1's_2', s_1't_2', s_2'u_1, s_2't_1', v_nt_1', u_nt_2', t_1't_2'$. Then $G'$ is isomorphic to $\trugr_{n+2}$, and $G'$ certainly has $G$ as linkage minor.

  Hence we may assume that each end vertex of $P_1$ and $P_2$ has degree at least two. Suppose no rung edge runs between two of these end vertices. Then it is not hard to see that $G$ has an $\XX$ minor, a contradiction. Therefore some two end vertices must be connected. By reversing paths as necessary, we may assume there is an edge $e = s_1s_2$.

By our assumption, $G\delete e$ can be embedded in  $\trugr_n$ for some $n$. Again, let $G'$ be obtained from $\trugr_n$ by adding four vertices $s_1', t_1', s_2', t_2'$, and edges $s_1'v_1, s_1's_2', s_1't_2', s_2'u_1, s_2't_1', v_nt_1', u_nt_2', t_1't_2'$. Then $G'$ is isomorphic to $\trugr_{n+2}$, and $G'$ certainly has $G$ as linkage minor, a contradiction.
\end{proof}

As an aside, it is possible to prove a stronger version of the previous lemma. We say a partition $(A,B)$ of the rung edges is \emph{valid} if the edges in $A$ are pairwise non-crossing, and the edges in $B$ are pairwise non-crossing after reversing one of the paths. One can show:
\begin{itemize}
	\item Each Truemper graph has a valid partition.
	\item For every valid partition $(A,B)$ of a Truemper graph $G$, some  $\trugr_n$ has $G$ as linkage minor in such a way that $(A,B)$ extends to a valid partition of $\trugr_n$.
\end{itemize}

Now we have all ingredients of our main result.

\begin{proof}[Proof of Theorem \ref{thm:mainresult}]
  From Lemma \ref{lem:i-ii} we learn that \eqref{main:i}$\Leftrightarrow$\eqref{main:ii}. From Lemma \ref{lem:iii-ii} we learn that \eqref{main:iii}$\Rightarrow$\eqref{main:ii}, and from Lemma 
 \ref{lem:signtoTru} we conclude that \eqref{main:ii}$\Rightarrow$\eqref{main:iii}.
\end{proof}

\renewcommand{\Dutchvon}[2]{#1}
\bibliography{matbib2009}
\bibliographystyle{plainnat}

\end{document}